 \documentclass[final,3p,sort&compress]{elsarticle}

\usepackage{epsfig}
\usepackage{amssymb}
\usepackage{amsmath}
\usepackage{amsthm}
\usepackage{amsbsy}
\usepackage{multirow}
\usepackage{graphicx}
\usepackage[pdftex,bookmarksnumbered]{hyperref}
\usepackage{multicol}

\newtheorem{theorem}{Theorem}[section]
\newtheorem{definition}{Definition}[section]
\newtheorem{lemma}{Lemma}[section]

\usepackage{setspace}

\usepackage{lineno}

\journal{Elsevier}

\begin{document}

\begin{frontmatter}



\title{Legendre-Fenchel duality and a generalized constitutive relation error}


\author[thu]{Mengwu Guo\corref{cor1}}
\ead{gmw13@mails.tsinghua.edu.cn}
\cortext[cor1]{Corresponding author.}

\author[Iowa]{Weimin Han}

\author[thu]{Hongzhi Zhong}

\address[thu]{Department of Civil Engineering, Tsinghua University, Beijing 100084, China}
\address[Iowa]{Department of Mathematics, University of Iowa, Iowa City, IA 52242, USA}

\begin{abstract}
  A generalized constitutive relation error is proposed in an analogous form to Fenchel-Young inequality on the basis of the key idea of Legendre-Fenchel duality theory. The generalized constitutive relation error is linked with the global errors of some admissible solutions for the problem in question, and is of wide applicability, especially in a posteriori error estimations of numerical methods. A class of elliptic variational inequalities is examined using the proposed approach and a strict upper bound of global energy errors of admissible solutions is obtained.

\end{abstract}




\end{frontmatter}


\section{Introduction}

Legendre-Fenchel duality \cite{ekeland1976convex,ciarlet2012new} is an important concept in convex analysis and variational methods, and acts as the theoretical basis of the link between primal and dual formulations of some convex optimization problems. In some PDE problems described in a variational form, Legendre-Fenchel tansforms are used to formulate the corresponding dual variational problems. A classic example is the complementary energy principle versus the potential energy principle for an elastic problem \cite{atkinson2005theoretical}.

The constitutive relation error (CRE) was proposed by Ladev\`{e}ze \cite{ladeveze1983error,ladeveze2005mastering} as an estimator of \emph{a posteriori} error of finite element analysis. Among the available techniques of \emph{a posteriori} error estimation, the CRE-based estimation provides guaranteed strict upper bounds of global energy-norm errors of numerical solutions of the PDEs \cite{ladeveze1983error}, and guaranteed strict upper and lower bounds of numerical errors in various quantities of interest \cite{ladeveze1999local}. The strict bounding property, together with its advantage of wide applicability, makes the CRE stand out for \emph{a posteriori} error estimation in a variety of practical problems \cite{chamoin2009strict,ladeveze2010calculation,ladeveze2008strict,panetier2009strict,ladeveze2013new,guo2015goal,wang2015unified}.

For the constitutive relations represented by a strain energy functional and complementary strain energy functional which are mutually convex conjugate, or referred to as Legendre-Fenchel transform, Ladev\`{e}ze introduced an expression of CRE in the form of Fenchel-Young inequality \cite{ladeveze2005mastering}, a natural result of Legendre-Fenchel duality \cite{borwein2010convex}. In this paper, Legendre-Fenchel duality is briefly reviewed with its application to the establishment of complementary energy principle from the potential energy principle for a hyperelastic problem. Similar to the CRE form for the hyperelastic materials, a generalized constitutive relation error (GCRE) is proposed in an analogous but generalized formulation of Fenchel-Young inequality. Then the GCRE is used in a class of elliptic variational inequalities. It is verified that the GCRE provides strict upper bounds of global errors of admissible solutions, which could be further applied to \emph{a posteriori} error estimation for the numerical solutions to these variational inequalities.

Following the introduction, some basic concepts about Legendre-Fenchel duality and its application to hyperelastic problem are introduced in Section 2. After an overview of the CRE for hyperelasticity, the GCRE is proposed in Section 3. A class of elliptic variational inequalities is dealt with using the GCRE and some significant properties are highlighted in Section 4. In Section 5, the application of the GCRE to \emph{a posteriori} error estimation is briefly discussed, and conclusions are drawn in Section 6.

\section{Legendre-Fenchel duality}

In this section, some important concepts about Legendre-Fenchel duality are briefly reviewed. The duality is also adopted for the transformation from potential energy principle to complementary  energy principle of a static hyperelastic problem.

\subsection{Basic concepts}

The duality between a normed space $Y$ and its dual space $Y'$ is denoted by  $_{Y'}\langle\cdot,\cdot\rangle_{Y}$. The dual space of $Y'$, denoted by $Y''$, is identified as $Y$ by means of the usual canonical isometry if $Y$ is a reflexive Banach space.

Throughout this paper, $I_{\{\bullet\}}$ denotes the indicator function of set $\{\bullet\}$, i.e. $I_{\{\bullet\}}(y)=0$ if $y\in {\{\bullet\}}$ and $I_{\{\bullet\}}(y)=+\infty$ if $y\notin {\{\bullet\}}$. A function $f:Y\to\mathbb{R}\cup\{+\infty\}$ is called a proper function if $\{y\in Y:f(y)<+\infty\}\neq \emptyset$.

Considering a proper function $f:Y\to\mathbb{R}\cup\{+\infty\}$ on a normed space $Y$, its Legendre-Fenchel transform $f^*:Y'\to\mathbb{R}\cup\{+\infty\}$ is defined as
\begin{equation}
f^*(y^*)=\sup_{y\in Y}\left\{_{Y'}\langle y^*,y \rangle_{Y}-f(y)\right\}\,,\quad y^*\in Y'\,.
\end{equation}
Thus Fenchel-Young inequality is written as follows:
\begin{equation}
f(y)+f^*(y^*)-_{Y'}\langle y^*,y \rangle_{Y}\geq 0\quad \forall(y,y^*)\in Y\times Y'\,.
\end{equation}

If $Y$ is a reflexive Banach space and $f:Y\to\mathbb{R}\cup\{+\infty\}$ is a proper, convex and lower semi-continuous function, the Legendre-Fenchel transform $f^*:Y'\to\mathbb{R}\cup\{+\infty\}$ of $f$ is also proper, convex and lower semi-continuous, and the Legendre-Fenchel transform $f^{**}$ of $f^*$ is identified as $f$, i.e. $f^{**}=f$.

Let $Y$ and $Z$ be two reflexive Banach spaces, $f:Y\to\mathbb{R}\cup\{+\infty\}$ and $h:Z'\to\mathbb{R}\cup\{+\infty\}$ be two proper, convex and lower semi-continuous functions, $\Lambda:Y\to Z'$ a linear continuous operator. Define $J:Y\to\mathbb{R}\cup\{+\infty\}$ as
\begin{equation}
J(y)=f(y)+h(\Lambda y)\,,\quad y\in Y\,,
\end{equation}
and $\tilde{L}:Y\times Z$ as
\begin{equation}
\tilde{L}(y,z)=f(y)+_{Z'}\langle\Lambda y,z \rangle_{Z}-h^*(z)\,,\quad (y,z)\in Y\times Z.
\end{equation}
From the definition of Legendre-Fenchel transform of $h^*$,
\begin{equation*}
h(z^*)=h^{**}(z^*)=\sup_{z\in Z}\left\{_{Z'}\langle z^*,z \rangle_{Z}-h(z)\right\}\,,\quad z^*\in Z'\,,
\end{equation*}
one has
\begin{equation}\label{min}
\inf_{y\in Y}~J(y)=\inf_{y\in Y}\sup_{z\in Z}~\tilde{L}(y,z)\,.
\end{equation}

If the minimizing problem in \eqref{min} is called the primal problem, its dual problem (or dual formulation) is given in a maximizing form as
\begin{equation}\label{max}
\sup_{z\in Z}~G(z)
\end{equation}
where the function $G:Z\to \mathbb{R}\cup\{-\infty\}$ is defined as
\begin{equation}
G(z)=\inf_{y\in Y}~\tilde{L}(y,z)\,,\quad z\in Z\,.
\end{equation}

What needs to be further determined is whether the infimum found in the primal problem  \eqref{min} is equal to the supremum found in the dual problem  \eqref{max}, i.e. whether the following relations hold true
\begin{equation}
\inf_{y\in Y}J(y)=\inf_{y\in Y}\sup_{z\in Z}~\tilde{L}(y,z)=\sup_{z\in Z}\inf_{y\in Y}~\tilde{L}(y,z)=\sup_{z\in Z}~G(z).
\end{equation}
This is equivalent to deciding whether a saddle point $(\bar{y},\bar{z})$ of $\tilde{L}$ exists in $Y\times Z$, i.e
\begin{equation}
\inf_{y\in Y}~\tilde{L}(y,\bar{z})=\tilde{L}(\bar{y},\bar{z})=\sup_{z\in Z}~\tilde{L}(\bar{y},{z})\,.
\end{equation}
If the saddle point $(\bar{y},\bar{z})$ exists, the infimum of \eqref{min}, the supremum of \eqref{max} and $\tilde{L}(\bar{y},\bar{z})$ are all equal.

\subsection{Legendre-Fenchel duality in hyperelasticity}

The definitions of the constitutive relations for hyperelastic materials are in the form of Legendre-Fenchel duality. To show specific applications of the duality, a static problem of a hyperelastic body is taken into consideration.

\subsubsection{Problem definition}
\begin{figure}
  \centering
  \includegraphics[scale=0.6]{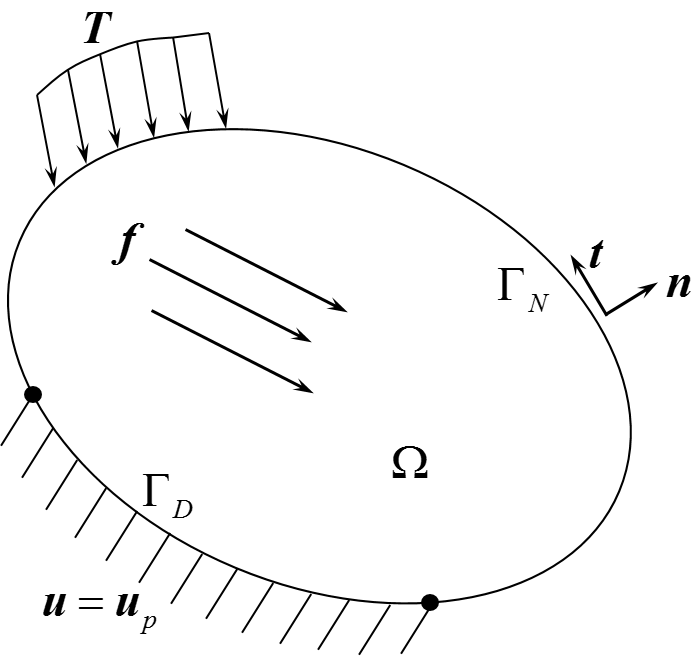}
  \caption{A static problem of an elastic body}\label{static}
\end{figure}

Consider an elastic body whose undeformed configuration $\boldsymbol{X}$ is defined in a domain $\Omega\subset \mathbb{R}^3$ with a Lipschitz boundary $\Gamma=\Gamma_D\cup \Gamma_N$, where $\Gamma_D\neq \emptyset$ is the Dirichlet boundary, $\Gamma_N$ is the Neumann boundary, and $\Gamma_D\cap\Gamma_N=\emptyset$. The body is subject to a prescribed body force of density $\boldsymbol{f}\in [L^2(\Omega)]^3$ in $\Omega$ with respect to the undeformed volume, a prescribed displacement $\boldsymbol{u}_D$ on $\Gamma_D$, and a prescribed surface force of density $\boldsymbol{T}\in [L^2(\Gamma_N)]^3$ with respect to the undeformed surface area.


Let us define a space $\mathcal{E}= [H(\mathrm{div},\Omega)]^{3}$, where $H(\mathrm{div},\Omega)=\{\boldsymbol{q}\in [L^2(\Omega)]^d:\mathrm{div}\boldsymbol{q}\in L^2(\Omega)\}$.
A continuous positive-definite bilinear form $\langle\cdot,\cdot\rangle:\mathcal{E}\times \mathcal{E} \to \mathbb{R}$, also a duality pair, is defined in the following form throughout this paper:
\begin{equation}
\langle \boldsymbol{\sigma},\boldsymbol{\varepsilon} \rangle=\int_{\Omega} \boldsymbol{\sigma}:\boldsymbol{\varepsilon}\,, \quad (\boldsymbol{\sigma},\boldsymbol{\varepsilon})\in \mathcal{E}\times \mathcal{E}\,.
\end{equation}

Consider the following problem: find $(\boldsymbol{u},\boldsymbol{P})$ such that
\begin{equation}\label{gov1}
\mathrm{div}\boldsymbol{P}+\boldsymbol{f}=\boldsymbol{0}\quad \mathrm{in}\;\Omega\,,\quad \boldsymbol{P}\boldsymbol{n}=\boldsymbol{T}\quad \mathrm{on}\;\Gamma_N\,,\quad \boldsymbol{u}=\boldsymbol{u}_D\quad \mathrm{on}\;\Gamma_D\,,
\end{equation}
together with a hyperelastic constitutive relation expressed in the following form:
\begin{equation}\label{gov2}
W(\boldsymbol{F})+W^*(\boldsymbol{P})-\langle \boldsymbol{P},\boldsymbol{F} \rangle = 0\,,
\end{equation}
where $\boldsymbol{u}$ is the vector of displacements, $\boldsymbol{P}$ is the first Piola-Kirchhoff stress tensor, $\boldsymbol{F}=\boldsymbol{F}(\boldsymbol{u})=\boldsymbol{I}+\nabla \boldsymbol{u}=\nabla \boldsymbol{\psi}$, $\boldsymbol{I}$ is the second-order unit tensor, $\boldsymbol{\psi}=\boldsymbol{X}+\boldsymbol{u}$ is the deformed configuration, and two proper functionals $W:\mathcal{E}\to\mathbb{R}$ and $W^*:\mathcal{E}\to\mathbb{R}$ are referred to as strain energy functional and complementary strain energy functional, respectively.  As the 'convex conjugate' or Legendre-Fenchel transform of $W$, $W^*$ is defined as
\begin{equation}\label{Wdual}
W^*(\boldsymbol{\tau})=\sup_{\boldsymbol{\epsilon}\in \mathcal{E}}\{\langle \boldsymbol{\tau},\boldsymbol{\epsilon} \rangle-W(\boldsymbol{\epsilon})\}\,.
\end{equation}

The Fenchel-Young inequality is then given as
\begin{equation}\label{CRE1}
W(\boldsymbol{\epsilon})+W^*(\boldsymbol{\tau})-\langle \boldsymbol{\tau},\boldsymbol{\epsilon} \rangle \geq 0\quad \forall (\boldsymbol{\epsilon},\boldsymbol{\tau})\in \mathcal{E}\times \mathcal{E}\,,
\end{equation}
and the constitutive relation is satisfied if and only if the equal sign holds true. Alternatively, the relation is expressed in the following subdifferential form:
\begin{equation}
\boldsymbol{P} \in \partial W(\boldsymbol{F})\,.
\end{equation}

The strain energy functional for hyperelastic materials is usually written as the integration of a stored strain energy function $\tilde{W}(\cdot,\cdot):\Omega\times \mathbb{M}_{+}^{3}\to \mathbb{R}$ over the domain $\Omega$, i.e.
\begin{equation}
W(\boldsymbol{F})=\int_{\Omega}\tilde{W}(\boldsymbol{x},\boldsymbol{F})\,,
\end{equation}
where $\mathbb{M}_{+}^{3}:=\{\boldsymbol{F}\in \mathbb{M}^3:\mathrm{det}\boldsymbol{F}>0\}$ with $\mathbb{M}^3$ being the set of all real square matrices of order 3. In order to ensure the existence of solution to this hyperelastic problem, which will be expressed in the form of minimization of a potential energy in subsection 2.2.2, the stored strain energy function $\tilde{W}$ is assumed to satisfy the following properties \cite{ciarlet1988mathematical}:
\begin{itemize}
  \item polyconvexity: for almost all $\boldsymbol{x}\in\Omega$, there exists a convex function $\mathbb{W}(\boldsymbol{x},\cdot,\cdot,\cdot):\mathbb{M}_{+}^{3}\times \mathbb{M}_{+}^{3}\times(0,+\infty)\to\mathbb{R}$ such that
      \begin{equation}
      \mathbb{W}(\boldsymbol{x},\boldsymbol{F},\mathrm{Cof}~\boldsymbol{F},\mathrm{det}~\boldsymbol{F})=\tilde{W}(\boldsymbol{x},\boldsymbol{F})\quad \forall \boldsymbol{F}\in \mathbb{M}^{3}_{+}\,,
      \end{equation}
      where $\mathrm{Cof}~\boldsymbol{F}:=(\mathrm{det}~\boldsymbol{F})\boldsymbol{F}^{-T}$, and the function $\mathbb{W}(\cdot,\boldsymbol{F},\boldsymbol{H},\delta):\Omega\to\mathbb{R}$ is measurable for all $(\boldsymbol{F},\boldsymbol{H},\delta)\in \mathbb{M}_{+}^{3}\times \mathbb{M}_{+}^{3}\times(0,+\infty)$;
  \item behavior as $\mathrm{det}~\boldsymbol{F}\to 0^+$: for almost all $\boldsymbol{x}\in \Omega$,
      \begin{equation}
      \lim_{\mathrm{det}~\boldsymbol{F}\to 0^+}\tilde{W}(\boldsymbol{x},\boldsymbol{F})=+\infty\,;
      \end{equation}
  \item coerciveness: there exist constants $\alpha$, $\beta$, $s$, $t$, $m$ such that $\alpha>0$, $s\geq 2$, $t\geq s/(s-1)$, $m>1$,
  \begin{equation}
  \tilde{W}(\boldsymbol{x},\boldsymbol{F})\geq \alpha(\|\boldsymbol{F}\|^s+\|\mathrm{Cof}~\boldsymbol{F}\|^t+(\mathrm{det}~\boldsymbol{F})^m)+\beta
  \end{equation}
  for almost $\boldsymbol{x}\in \Omega$ and $\forall \boldsymbol{F}\in \mathbb{M}^3_{+}$, where $\|\cdot\|$ is the matrix norm of $\mathbb{M}^3$, i.e. $\|\boldsymbol{A}\|=(\boldsymbol{A}:\boldsymbol{A})^{\frac{1}{2}}$, $\boldsymbol{A}\in \mathbb{M}^3$.
\end{itemize}

The energy functional $W^*$ is the Fenchel conjugate of $W$, meaning that $W^*$ is convex and lower semi-continuous, and it can be rewritten by a stored complementary strain energy function $\tilde{W}^*(\cdot,\cdot):\Omega\times \mathbb{M}^3\to \mathbb{R}$ as
\begin{equation}
W^*(\boldsymbol{\tau})=\int_{\Omega}\tilde{W}^*(\boldsymbol{x},\boldsymbol{\tau})\,,
\end{equation}
with $\tilde{W}^*$ being the Fenchel conjugate of $\tilde{W}$, i.e. for almost all $\boldsymbol{x}\in \Omega$
\begin{equation}
\tilde{W}^*(\boldsymbol{x},\boldsymbol{\tau})=\sup_{\boldsymbol{\epsilon}\in \mathbb{M}^3}\{\boldsymbol{\tau}:\boldsymbol{\epsilon}-\tilde{W}(\boldsymbol{x},\boldsymbol{\epsilon})\}\,.
\end{equation}
In order to assure the existence of solution to the dual variational problem, which will be introduced in 2.2.2, the coerciveness of function $\tilde{W}^*$ is assumed as: there exist constants $\alpha '$, $\beta '$, $s'$ such that $\alpha '>0$, $s>1$,
\begin{equation}
\tilde{W}^*(\boldsymbol{x},\boldsymbol{\tau})\geq \alpha '\|\boldsymbol{\tau}\|^{s'}+\beta'
\end{equation}
for almost $\boldsymbol{x}\in \Omega$ and $\forall \boldsymbol{\tau}\in \mathbb{M}^3$.

\subsubsection{Energy principles and duality}

For notation, two sets are defined as
\begin{equation}
\begin{split}
X_D=&\{\boldsymbol{\psi}\in[W^{1,s}(\Omega)]^3:\mathrm{Cof}~\nabla\boldsymbol{\psi}\in[L^t(\Omega)]^{3\times 3},\mathrm{det}~\nabla\boldsymbol{\psi}\in L^{m}(\Omega),\\
& ~~~~\boldsymbol{\psi}=\boldsymbol{X}+\boldsymbol{u}_D~\mathrm{a.e.~on}~\Gamma_D,\mathrm{det}~\nabla\boldsymbol{\psi}>0~\mathrm{a.e.~in}~\Omega\}-\boldsymbol{X}\,,\\
X=&\{\boldsymbol{\psi}\in[W^{1,s}(\Omega)]^3:\mathrm{Cof}~\nabla\boldsymbol{\psi}\in[L^t(\Omega)]^{3\times 3},\mathrm{det}~\nabla\boldsymbol{\psi}\in L^{m}(\Omega),\\
& ~~~~\mathrm{det}~\nabla\boldsymbol{\psi}>0~\mathrm{a.e.~in}~\Omega\}-\boldsymbol{X}\,,
\end{split}
\end{equation}
and two linear functionals $l_p\in X'$ and $l_c\in \mathcal{E}'$ are defined as
\begin{equation}
l_p(\boldsymbol{v})=\int_\Omega\boldsymbol{f}\cdot\boldsymbol{v}+\int_{\Gamma_N}\boldsymbol{T}\cdot\boldsymbol{v}\,,\quad l_c(\boldsymbol{\tau})=\int_{\Gamma_D}(\boldsymbol{\tau}\boldsymbol{n})\cdot\boldsymbol{u}_D \,,\quad (\boldsymbol{v},\boldsymbol{\tau})\in X\times \mathcal{E}.
\end{equation}

The model problem under consideration is alternatively described in a (primal) variational form:
\begin{equation}
\begin{split}
& \Pi_p(\boldsymbol{u})=\min_{\boldsymbol{v}\in X_D}\Pi_p(\boldsymbol{v})\\
& \Pi_p:X\to \mathbb{R},\boldsymbol{v}\mapsto W(\boldsymbol{F}(\boldsymbol{v}))-l_p(\boldsymbol{v})\,,
\end{split}
\end{equation}
where $\Pi_p$ is usually called the potential energy functional, and this variational formulation, referred to as potential energy principle, can be rewritten as
\begin{equation}\label{primalhyp}
\begin{split}
& J(\boldsymbol{u})=\min_{\boldsymbol{v}\in X}J(\boldsymbol{v})\\
& J:X\to \mathbb{R}\cup \{+\infty\},\boldsymbol{v}\mapsto \Pi_p(\boldsymbol{v})+I_{X_D}(\boldsymbol{v})\,.
\end{split}
\end{equation}
The existence of the minimizer $\boldsymbol{u}\in X$ has been ensured.

From the duality in \eqref{Wdual}, the minimizing problem in \eqref{primalhyp} is further expressed as
\begin{equation}
\begin{split}
& \qquad\qquad\qquad\min_{\boldsymbol{v}\in X}J(\boldsymbol{v})=\min_{\boldsymbol{v}\in X}\max_{\boldsymbol{\tau}\in\mathcal{E}}~\tilde{\mathcal{L}}(\boldsymbol{v},\boldsymbol{\tau})\\
& \tilde{\mathcal{L}}(\boldsymbol{v},\boldsymbol{\tau})=\langle\boldsymbol{\tau}, \boldsymbol{F}(\boldsymbol{v})\rangle -W^*(\boldsymbol{\tau})-l_p(\boldsymbol{v})+I_{X_D}(\boldsymbol{v})
\,,\quad (\boldsymbol{v},\boldsymbol{\tau})\in X\times \mathcal{E}\,,
\end{split}
\end{equation}
which is exactly the treatment given in \eqref{min}.

\begin{theorem}
The dual variational problem of \eqref{primalhyp} is expressed as
\begin{equation}
\begin{split}
& G(\bar{\boldsymbol{P}})=\max_{\boldsymbol{\tau}\in\mathcal{E}}~ G(\boldsymbol{\tau})\\
& G:\mathcal{E}\to \mathbb{R}\cup\{-\infty\},\boldsymbol{\tau}\mapsto -W^*(\boldsymbol{\tau})+\int_{\Omega}\mathrm{tr}\boldsymbol{\tau}+l_c(\boldsymbol{\tau})-I_{SA}(\boldsymbol{\tau})\,,
\end{split}
\end{equation}
where $\mathrm{tr}\boldsymbol{\tau}:=\boldsymbol{\tau}:\boldsymbol{I}$, $SA$ is the so-called statically admissible set of stress fields, i.e.
\begin{equation}
SA=\{\boldsymbol{\tau}\in \mathcal{E}:\mathrm{div}\boldsymbol{\tau}+\boldsymbol{f}=0~\;\mathrm{a.e.~in}\;\Omega\,,\;
\boldsymbol{\tau}\boldsymbol{n}=\boldsymbol{T}~\;\mathrm{a.e.~on}\;\Gamma_N\}\,.
\end{equation}
\end{theorem}

\begin{proof}
With Green's theorem, one has
\begin{equation}\label{Green}
\langle\boldsymbol{\tau}, \nabla \boldsymbol{v}\rangle=\int_{\Omega}-\mathrm{div}\boldsymbol{\tau}\cdot\boldsymbol{v}+\int_{\Gamma}(\boldsymbol{\tau}\boldsymbol{n})\cdot\boldsymbol{v}\quad (\boldsymbol{v},\boldsymbol{\tau})\in X\times \mathcal{E}\,,
\end{equation}
and therefore
\begin{equation}
\begin{split}
G(\boldsymbol{\tau})& =\min_{\boldsymbol{v}\in X}~\tilde{\mathcal{L}}(\boldsymbol{v},\boldsymbol{\tau})\\
& =\min_{\boldsymbol{v}\in X_D}~\left\{\langle\boldsymbol{\tau}, \boldsymbol{I}+\nabla \boldsymbol{v}\rangle -W^*(\boldsymbol{\tau})-l_p(\boldsymbol{v})\right\}\\
& =\min_{\boldsymbol{v}\in X_D}~\left\{\int_{\Omega}-(\mathrm{div}\boldsymbol{\tau}+\boldsymbol{f})\cdot\boldsymbol{v}+\int_{\Gamma_N}(\boldsymbol{\tau}\boldsymbol{n}-\boldsymbol{T})\cdot\boldsymbol{v}+\int_{\Omega}\mathrm{tr}\boldsymbol{\tau}+l_c(\boldsymbol{\tau})-W^*(\boldsymbol{\tau})\right\}\\
& =-W^*(\boldsymbol{\tau})+\int_{\Omega}\mathrm{tr}\boldsymbol{\tau}+l_c(\boldsymbol{\tau})-I_{SA}(\boldsymbol{\tau})\,,\quad \boldsymbol{\tau}\in\mathcal{E}\,.
\end{split}
\end{equation}
Since $G$ is coercive, concave and upper semi-continuous on reflexive Banach space $\mathcal{E}$, the existence of maximizer $\bar{\boldsymbol{P}}$ is ensured.
\end{proof}

The dual variational formulation, referred to as complementary energy principle, can be rewritten as
\begin{equation}
\begin{split}
& \Pi_c(\bar{\boldsymbol{P}})=\min_{\boldsymbol{\tau}\in SA}\Pi_c(\boldsymbol{\tau})\\
& \Pi_c:\mathcal{E}\to \mathbb{R},\boldsymbol{\tau}\mapsto W^*(\boldsymbol{\tau})-\int_{\Omega}\mathrm{tr}\boldsymbol{\tau}-l_c(\boldsymbol{\tau})\,,
\end{split}
\end{equation}
where $\Pi_c$ is usually called the complementary energy functional.

\begin{lemma}\label{lplc}
For $(\boldsymbol{v},\boldsymbol{\tau}) \in X_D\times SA$,
\begin{equation}\label{lclp}
\langle\boldsymbol{\tau}, \nabla \boldsymbol{v}\rangle=l_p(\boldsymbol{v})+l_c(\boldsymbol{\tau})\,.
\end{equation}
\end{lemma}

\begin{proof}
This comes out naturally from Green's theorem \eqref{Green} and the definitions of $X_D$ and $SA$.
\end{proof}

\begin{theorem}
The stress solution $\boldsymbol{P}$ of the model problem stated by \eqref{gov1} and \eqref{gov2} is a maximizer of $G$ in $\mathcal{E}$.
\end{theorem}

\begin{proof}
All needed to prove is $\boldsymbol{P}\in SA$ is exactly a minimizer of $\Pi_c$ in $SA$, which can be completed once the following identity, an extreme condition
\begin{equation}
\langle \boldsymbol{\tau}-\boldsymbol{P},\boldsymbol{I}+\nabla \boldsymbol{u}\rangle-\int_{\Omega}\mathrm{tr}(\boldsymbol{\tau}-\boldsymbol{P})-l_c(\boldsymbol{\tau}-\boldsymbol{P})=0 \quad \forall \boldsymbol{\tau}\in SA
\end{equation}
is verified.

This identity can be directly obtained from Lemma \ref{lplc}.
\end{proof}

\begin{theorem}
$(\boldsymbol{u},\boldsymbol{P})$ is a saddle point of $\tilde{\mathcal{L}}$ in $X\times\mathcal{E}$, i.e.
\begin{equation}
\min_{\boldsymbol{v}\in X}\max_{\boldsymbol{\tau}\in\mathcal{E}}~\tilde{\mathcal{L}}(\boldsymbol{v},\boldsymbol{\tau})=\tilde{\mathcal{L}}(\boldsymbol{u},\boldsymbol{P})=\max_{\boldsymbol{\tau}\in\mathcal{E}}\min_{\boldsymbol{v}\in X}~\tilde{\mathcal{L}}(\boldsymbol{v},\boldsymbol{\tau})\,.
\end{equation}
\end{theorem}

\begin{proof}
Evidently,
\begin{equation*}
\begin{split}
& \min_{\boldsymbol{v}\in X}\max_{\boldsymbol{\tau}\in\mathcal{E}}~\tilde{\mathcal{L}}(\boldsymbol{v},\boldsymbol{\tau})=\min_{\boldsymbol{v}\in X}~ J(\boldsymbol{v})=\Pi_p(\boldsymbol{u})\,,\\
& \max_{\boldsymbol{\tau}\in\mathcal{E}}\min_{\boldsymbol{v}\in X}~\tilde{\mathcal{L}}(\boldsymbol{v},\boldsymbol{\tau})=\max_{\boldsymbol{\tau}\in\mathcal{E}}~G(\boldsymbol{\tau})=-\Pi_c(\boldsymbol{P})\,.
\end{split}
\end{equation*}
With the constitutive relation \eqref{gov2} and the identity \eqref{lclp}, one has
\begin{equation}\label{pippic}
\Pi_p(\boldsymbol{u})=\tilde{\mathcal{L}}(\boldsymbol{u},\boldsymbol{P})=-\Pi_c(\boldsymbol{P})\,.
\end{equation}
\end{proof}

\section{A generalized constitutive relation error}

In this section, a generalized constitutive relation error (GCRE) is defined, which is applicable to a wide range of problems. Some important properties of the GCRE are obtained. The constitutive relation error (CRE) for the hyperelastic problem is also discussed as a special case.

\subsection{Definition of GCRE}
\begin{definition}[GCRE]
A generalized constitutive relation error (GCRE) $\Psi:\mathcal{U}\times\mathcal{T}\to\mathbb{R}$ is defined as follows
\begin{equation}\label{gcre}
 \Psi(\hat{U},\hat{P})=\phi(\hat{U})+\phi^*(\hat{P})-\mathcal{B}(\hat{U},\hat{P})\geq 0\,, \quad (\hat{U},\hat{P})\in \mathcal{U}\times\mathcal{T}\,,
\end{equation}
where $\mathcal{U}\subset\mathcal{H}$ and $\mathcal{T}\subset\mathcal{S}$ are closed convex nonempty subsets of reflexive Banach spaces $\mathcal{H}$ and $\mathcal{S}$;  $\mathcal{B}(\cdot,\cdot):\mathcal{H}\times \mathcal{S}\to \mathbb{R}$ is a continuous bilinear form; $\phi:\mathcal{H}\to \mathbb{R}\cup\{+\infty\}$ and $\phi^*:\mathcal{S}\to \mathbb{R}\cup\{+\infty\}$ are generalized potential and complementary energy functionals, respectively, both being convex, lower semi-continuous and proper. The equality $\Psi(\hat{U},\hat{P})=0$ is equivalent to the satisfaction of the constitutive relation.
\end{definition}

The definition of the GCRE can be considered as a natural generalization of the form of Fenchel-Young inequality $\eqref{CRE1}$ for hyperelasticity. Suppose that the governing equations of the problem can be given in the following form: find $({U},{P})\in \mathcal{H}\times \mathcal{S}$ such that
\begin{equation}
U\in \mathcal{U}\,,\quad P\in\mathcal{T}\,, \quad \Psi({U},{P})=0\,,
\end{equation}
which are called compatibility condition, equilibrium condition and constitutive relation, respectively. The existence and uniqueness of the exact solution pair $(U,P)$ is assumed to be ensured.

In order to evaluate the errors between admissible fields $(\hat{U},\hat{P})$ and exact fields $({U},{P})$, two error functionals are defined.

\begin{definition}[Error functionals]\label{errorfunc}
Error functionls $\bar{\phi}_U:\mathcal{H}\to \mathbb{R}$ and $\bar{\phi}^*_P:\mathcal{S}\to \mathbb{R}$ are introduced as follows:
\begin{equation}\label{errdef}
\begin{split}
& \bar{\phi}_{(U,P)}(e)=\phi(U+e)-\phi(U)-\mathcal{B}(e,P)\,,\quad  e\in \mathcal{H}\,,\\
& \bar{\phi}^*_{(U,P)}(r)=\phi^*(P+r)-\phi^*(P)-\mathcal{B}(U,r)\,,\quad  r \in \mathcal{S}\,,
\end{split}
\end{equation}
where $({U},{P})$ is the exact solution pair of the problem under consideration.
\end{definition}

\begin{theorem}
The error functionals $\bar{\phi}_{(U,P)}$ and $\bar{\phi}^*_{(U,P)}$ have the following properties:
\begin{equation}\label{assum}
\bar{\phi}_{(U,P)}(\hat{U}-U)\geq 0\,,\quad \bar{\phi}^*_{(U,P)}(\hat{P}-P)\geq 0\,,\quad \forall(\hat{U},\hat{P})\in\mathcal{U}\times\mathcal{T}\,.
\end{equation}
\end{theorem}

\begin{proof}
\eqref{gcre} gives
\begin{equation*}
\begin{split}
& \Psi(\hat{U},{P})\geq \Psi({U},{P})=0\quad \forall \hat{U}\in\mathcal{U}\,,\\
& \Psi({U},\hat{P})\geq \Psi({U},{P})=0\quad \forall \hat{P}\in\mathcal{T}\,,
\end{split}
\end{equation*}
from which \eqref{assum} is obtained directly.
\end{proof}

\subsection{The relation between GCRE and error functionals}
\begin{theorem}
The GCRE and the error functionals in Definition \ref{errorfunc} are interlinked by the following identity:
\begin{equation}\label{relation}
\Psi(\hat{U},\hat{P})=\bar{\phi}_{(U,P)}(\hat{U}-U)+\bar{\phi}_{(U,P)}^*(\hat{P}-P)-\mathcal{B}(\hat{U}-U,\hat{P}-P)\,,\quad (\hat{U},\hat{P})\in \mathcal{U}\times\mathcal{T}\,.
\end{equation}
\end{theorem}

\begin{proof}
From the definitions of error functionals \eqref{errdef}, one has
\begin{equation*}
\bar{\phi}_{(U,P)}(\hat{U}-U)+\bar{\phi}_{(U,P)}^*(\hat{P}-P)=  (\phi(\hat{U})+\phi^*(\hat{P}))- (\phi({U})+\phi^*({P}))
 -\mathcal{B}(\hat{U}-U,P)-\mathcal{B}(U,\hat{P}-P)\,.
\end{equation*}
Considering the fact that
\begin{equation*}
\phi(\hat{U})+\phi^*(\hat{P})= \Psi(\hat{U},\hat{P})+\mathcal{B}(\hat{U},\hat{P})
\end{equation*}
resulting from \eqref{gcre} and that
\begin{equation*}
\phi({U})+\phi^*({P})= \mathcal{B}({U},{P})
\end{equation*}
since the exact solution pair $(U,P)$ satisfies the constitutive relation, i.e. $\Psi({U},{P})=0$, regrouping the terms of bilinear form $\mathcal{B}$ yields the identity \eqref{relation} directly.
\end{proof}

\subsection{Overview of the CRE for hyperelasticity}

As a special case of the GCRE, the CRE for hyperelasticity is discussed in this subsection. The CRE form of hyperelasticity is naturally defined on $X\times \mathcal{E}$ as the Fenchel-Young inequality \eqref{CRE1}, i.e.
\begin{equation}\label{CREhyp}
\Psi(\boldsymbol{v},\boldsymbol{\tau})=W(\boldsymbol{F}(\boldsymbol{v}))+W^*(\boldsymbol{\tau})-\langle
\boldsymbol{\tau},\boldsymbol{F}(\boldsymbol{v})\rangle \geq 0\,,\quad  (\boldsymbol{v},\boldsymbol{\tau})\in X\times \mathcal{E}\,,
\end{equation}
and one can take
\begin{equation}
\phi(\boldsymbol{v})=W(\boldsymbol{F}(\boldsymbol{v}))\,,\quad \phi^*(\boldsymbol{\tau})=W^*(\boldsymbol{\tau})-\int_{\Omega}\mathrm{tr}\boldsymbol{\tau}\,,\quad
\mathcal{B}(\boldsymbol{v},\boldsymbol{\tau})=\langle \boldsymbol{\tau},\nabla \boldsymbol{v} \rangle\,,\quad  (\boldsymbol{v},\boldsymbol{\tau})\in X\times \mathcal{E}\,.
\end{equation}

The model problem stated by \eqref{gov1} and \eqref{gov2} is then simply written as
\begin{equation}
\boldsymbol{u}\in {X_D}\,,\quad \boldsymbol{P}\in {SA}\,, \quad \Psi(\boldsymbol{u},\boldsymbol{P})=0\,.
\end{equation}

The error functionals are correspondingly introduced as
\begin{equation}
\begin{split}
& \bar{\phi}_{(\boldsymbol{u},\boldsymbol{P})}(\boldsymbol{e})=W(\boldsymbol{F}(\boldsymbol{u}+\boldsymbol{e}))-W(\boldsymbol{F}(\boldsymbol{u}))-\langle
\boldsymbol{P},\nabla \boldsymbol{e}\rangle\,,\quad  \boldsymbol{e}\in X\,,\\
& \bar{\phi}^*_{(\boldsymbol{u},\boldsymbol{P})}(\boldsymbol{r})=W^*(\boldsymbol{P}+\boldsymbol{r})-W^*(\boldsymbol{P})-\langle
\boldsymbol{r},\boldsymbol{F}( \boldsymbol{u})\rangle\,,\quad  \boldsymbol{r}\in \mathcal{E}\,.
\end{split}
\end{equation}

\begin{theorem}
The CRE for hyperelasticity can be represented as the sum of two error functionals of the admissible solutions $(\hat{\boldsymbol{u}},\hat{\boldsymbol{P}})\in X_D\times SA$, i.e.
\begin{equation}\label{hypsplit}
\Psi(\hat{\boldsymbol{u}},\hat{\boldsymbol{P}})=  \bar{\phi}_{(\boldsymbol{u},\boldsymbol{P})}(\hat{\boldsymbol{u}}-\boldsymbol{u})+\bar{\phi}^*_{(\boldsymbol{u},\boldsymbol{P})}(\hat{\boldsymbol{P}}-\boldsymbol{P})\geq
\begin{cases}
\bar{\phi}_{(\boldsymbol{u},\boldsymbol{P})}(\hat{\boldsymbol{u}}-\boldsymbol{u})\,,\\
\bar{\phi}^*_{(\boldsymbol{u},\boldsymbol{P})}(\hat{\boldsymbol{P}}-\boldsymbol{P})\,.
\end{cases}
\end{equation}
\end{theorem}

\begin{proof}
Eq.\ \eqref{relation} gives the following identity for this hyperelastic problem:
\begin{equation}
\Psi(\hat{\boldsymbol{u}},\hat{\boldsymbol{P}})= \bar{\phi}_{(\boldsymbol{u},\boldsymbol{P})}(\hat{\boldsymbol{u}}-\boldsymbol{u})+\bar{\phi}^*_{(\boldsymbol{u},\boldsymbol{P})}(\hat{\boldsymbol{P}}-\boldsymbol{P})
 -\langle
\hat{\boldsymbol{P}}-\boldsymbol{P},\nabla (\hat{\boldsymbol{u}}-\boldsymbol{u})\rangle\,.
\end{equation}

Since $(\hat{\boldsymbol{u}},\boldsymbol{u})\in X_D \times X_D$ and  $(\hat{\boldsymbol{P}},\boldsymbol{P})\in SA\times SA$, \eqref{Green} gives
\begin{equation*}
\langle
\hat{\boldsymbol{P}},\nabla (\hat{\boldsymbol{u}}-\boldsymbol{u})\rangle
=\langle
\boldsymbol{P},\nabla (\hat{\boldsymbol{u}}-\boldsymbol{u})\rangle
=l_p(\hat{\boldsymbol{u}}-\boldsymbol{u})\,,
\end{equation*}
then \eqref{hypsplit} is obtained in view of the fact that $\bar{\phi}_{(\boldsymbol{u},\boldsymbol{P})}(\hat{\boldsymbol{u}}-\boldsymbol{u})\geq 0$ and $\bar{\phi}^*_{(\boldsymbol{u},\boldsymbol{P})}(\hat{\boldsymbol{P}}-\boldsymbol{P})\geq 0$.
\end{proof}

\vspace{2mm}
\noindent \textbf{Remark 3.1.} In the case of linear elasticity under small deformation, the difference between deformed and undeformed configurations is not taken into consideration, and the first Piola-Kirchhoff stress tensor coincides with the Cauchy stress tensor, i.e. $\boldsymbol{P}=\boldsymbol{\sigma}$. Then one has
\begin{equation}
W(\boldsymbol{F})=\frac{1}{2}\langle\boldsymbol{K}:(\boldsymbol{F}-\boldsymbol{I}),(\boldsymbol{F}-\boldsymbol{I}) \rangle\,,
\quad W^*(\boldsymbol{\tau})=\frac{1}{2}\langle\boldsymbol{\tau},\boldsymbol{K}^{-1}:\boldsymbol{\tau} \rangle+\int_{\Omega}\mathrm{tr}\boldsymbol{\tau}\,,\quad
(\boldsymbol{F},\boldsymbol{\tau})\in \mathcal{E}\times \mathcal{E}\,,
\end{equation}
and
\begin{equation}
\bar{\phi}_{(\boldsymbol{u},\boldsymbol{\sigma})}(\boldsymbol{e})=\frac{1}{2}\langle\boldsymbol{K}:\nabla \boldsymbol{e},\nabla \boldsymbol{e}\rangle\,,\quad
\bar{\phi}^*_{(\boldsymbol{u},\boldsymbol{\sigma})}(\boldsymbol{\theta})=\frac{1}{2}\langle\boldsymbol{\theta},\boldsymbol{K}^{-1}:\boldsymbol{\theta} \rangle \,,\quad
(\boldsymbol{e},\boldsymbol{\theta})\in X\times \mathcal{E}\,.
\end{equation}
The CRE has the conventional form as
\begin{equation}
\begin{split}
\Psi(\hat{\boldsymbol{u}},\hat{\boldsymbol{\sigma}})& =\frac{1}{2}\langle\boldsymbol{K}:\nabla\hat{\boldsymbol{u}},\nabla\hat{\boldsymbol{u}} \rangle
+\frac{1}{2}\langle\hat{\boldsymbol{\sigma}},\boldsymbol{K}^{-1}:\hat{\boldsymbol{\sigma}} \rangle-\langle\hat{\boldsymbol{\sigma}},\nabla\hat{\boldsymbol{u}}\rangle\\
& =\frac{1}{2}\langle\hat{\boldsymbol{\sigma}}-\boldsymbol{K}:\nabla\hat{\boldsymbol{u}},~\boldsymbol{K}^{-1}:\hat{\boldsymbol{\sigma}}-\nabla\hat{\boldsymbol{u}} \rangle \,,\quad
(\hat{\boldsymbol{u}},\hat{\boldsymbol{\sigma}})\in X_D\times SA\,.
\end{split}
\end{equation}
\vspace{2mm}

\begin{theorem}
The CRE for hyperelasticity can be represented as the sum of the potential energy subject to an admissible displacement field $\hat{\boldsymbol{u}}\in X_D$ and the complementary energy subject to an admissible stress field $\hat{\boldsymbol{\sigma}}\in SA$, i.e.
\begin{equation}
\Psi(\hat{\boldsymbol{u}},\hat{\boldsymbol{P}})=\Pi_p(\hat{\boldsymbol{u}})+\Pi_c(\hat{\boldsymbol{P}})\,,\quad (\hat{\boldsymbol{u}},\hat{\boldsymbol{P}})\in X_D\times SA\,.
\end{equation}
\end{theorem}

\begin{proof}
From the definition of CRE in \eqref{CREhyp}, this theorem comes out naturally from Lemma \ref{lplc}.
\end{proof}

\vspace{2mm}
\noindent \textbf{Remark 3.2.} When $\hat{\boldsymbol{u}}={\boldsymbol{u}}$ and $\hat{\boldsymbol{P}}={\boldsymbol{P}}$, one has $\Psi({\boldsymbol{u}},{\boldsymbol{P}})=\Pi_p({\boldsymbol{u}})+\Pi_c({\boldsymbol{P}})=0$, which agrees with
\eqref{pippic}.

\section{Application to some elliptic variational inequalities}

In this section, a class of elliptic variational inequalities (EVIs) is introduced and described by two variational formulations which can be viewed as Legendre-Fenchel dual problems. A GCRE is then defined for these EVI problems, showing some significant properties. A frictional contact problem of a linear elastic body is also presented as a typical example of these problems.

\subsection{Problem definition}

Let us consider the problems described by EVIs in the following (primal) variational formulation:
\begin{equation}\label{func0}
u=\arg\min_{v\in K}\left\{\frac{1}{2}[\mathcal{A}v,\mathcal{A}v]-l(v)+j(v)\right\}
\end{equation}
where $K\subset V$ is a closed convex nonempty set in a real Hilbert space $V$, $[\cdot,\cdot]$ is the inner product of a Hilbert space $S$, $\mathcal{A}:V\to S$ is a linear differential operator, $l(\cdot)$ is a continuous linear form on $V$, i.e. $l\in V'$, and $j(\cdot)$ is a lower semi-continuous convex functional defined on $V$, but not necessarily differentiable. Let us define a bilinear form $a(\cdot,\cdot):V\times V\to \mathbb{R}$ as $a(u,v)=[ \mathcal{A}u,\mathcal{A}v]$, $u,v\in V$, and assume that $a$ is continuous and coercive. Then the variational formulation \eqref{func0} can be rewritten as
\begin{equation}\label{func}
u=\arg \min_{v\in K}\left\{\frac{1}{2}a(v,v)-l(v)+j(v)\right\},
\end{equation}
or equivalently written as: $u\in K$ such that
\begin{equation}\label{weak}
a(u,v-u)+j(v)-j(u)\geq l(v-u)\quad \forall v\in K.
\end{equation}

Suppose that the convex set $K$ can be defined in terms of a convex cone $M$, which is a subset of a real Hilbert space $L$ and has its vertex at $\theta_L$ (zero element of $L$), and a linear form $g_1 \in L'$, i.e.
\begin{equation}\label{defK}
K=\{v\in V: b_1(v,\eta)\geq g_1(\eta),\forall \eta \in M\}\,,
\end{equation}
where $b_1(\cdot,\cdot)$ is a continuous bilinear form on $V\times L$. Assume that functional $j$ is defined in terms of a bounded subset $N$ of a real Hilbert space $Q$, such that
\begin{equation}\label{defj}
j(v)=\max_{\xi \in N}\left\{-b_2(v,\xi)+g_2(\xi)\right\}\,,\quad v\in V,
\end{equation}
in which $b_2(\cdot,\cdot)$ is a continuous bilinear form on $V\times Q$ and $g_2\in Q'$. Suppose that bilinear forms $b_1$ and $b_2$ satisfy the following condition: $(\eta,\xi)\in L \times Q$,
\begin{equation}\label{bcondition}
b_1(v,\eta)+b_2(v,\xi)=0\quad  \forall v\in V\quad \Leftrightarrow \quad(\eta,\xi)=(\theta_L,\theta_Q)\,.
\end{equation}

With the definition of $K$, an inequality constraint in \eqref{defK}, and that of functional $j$ in \eqref{defj} and introduction of two Langrangian multipliers, a functional is defined as:
\begin{equation}
\mathcal{L}(v,\eta,\xi)=\frac{1}{2}a(v,v)-l(v)-b(v,\eta)+ g_1(\eta)-b(v,\xi)+ g_2(\xi)\,.
\end{equation}
There exists only one saddle point of $\mathcal{L}$ in $V\times M\times N$, denoted by $(u,\lambda,\omega)$, i.e.
\begin{equation}\label{Lfunc}
\mathcal{L}(u,\lambda,\omega)=\min_{v\in V} \max_{\eta\in M,\xi\in N} \mathcal{L}(v,\eta,\xi)\,,
\end{equation}
or equivalently, $(u,\lambda,\omega)\in V\times M\times N$ such that
\begin{subequations}  \label{mix}
\begin{align}
a(u,v)-b_1(v,\lambda)-b_2(v,\omega)=l(v) \quad \forall v\in V\,,            \label{mix-1} \\
b_1(u,\eta-\lambda)\geq g_1(\eta-\lambda) \quad \forall \eta\in M\,,         \label{mix-2} \\
b_2(u,\xi-\omega)\geq g_2(\xi-\omega) \quad \forall \xi \in N\,,              \label{mix-3}
\end{align}
\end{subequations}
where $u$ can be verified to be the solution of minimization problem \eqref{func}.

Since $M$ is a convex cone, \eqref{mix-2} can be equivalently split into $b_1(u,\eta)\geq g_1(\eta)$, $\forall \eta \in M$, and $b_1(u,\lambda)= g_1(\lambda)$. The following lemma also comes into being.

\begin{lemma}
\eqref{mix-3} is equivalent to
\begin{equation}\label{IIr0}
b_2(u,\omega)-g_2(\omega)+j(u)=0.
\end{equation}
\end{lemma}
\begin{proof}
For $\xi \in N$, $b_2(u,\xi)-g_2(\xi)\geq b_2(u,\omega)-g_2(\omega)$ implies $b(u,\omega)-g_2(\omega)=\min_{\xi\in N}b_2(u,\xi)-g_2(\xi)=-j(u)$. Conversely, assume \eqref{IIr0}, then $-b_2(u,\omega)+g_2(\omega)=j(u)=\max_{\mu \in N}-b_2(u,\mu)+g_2(\mu)\geq -b_2(u,\xi)+g_2(\xi)$, $\forall \xi \in N$, hence \eqref{mix-3} holds.
\end{proof}

In summary, problem \eqref{mix} can be rewritten as a group of three types of equations:
\begin{itemize}
\item compatibility condition:
\begin{equation}\label{evic}
u \in K\,;
\end{equation}
\item equilibrium condition:
\begin{equation}\label{equi}
(p,\lambda,\omega) \in T\,,
\end{equation}
\begin{equation*}
T:=\{(q,\eta,\xi)\in S\times M\times N: [ p,\mathcal{A}v ] =b_1(v,\eta)
+b_2(v,\xi)+l(v),\forall v\in V  \}\,;
\end{equation*}
\item constitutive relations:
\begin{equation}\label{consti}
p=\mathcal{A}u\,,\quad b_1(u,\lambda)-g_1(\lambda)=0\,,\quad b_2(u,\omega)-g_2(\omega)+j(u)=0\,,
\end{equation}
or alternatively expressed as
\begin{equation}
\frac{1}{2}[p-\mathcal{A}u,p-\mathcal{A}u]+b_1(u,\lambda)-g_1(\lambda)+b_2(u,\omega)-g_2(\omega)+j(u)=0\,,
\end{equation}
for $(u,\lambda,\omega)\in K\times M\times N$.
\end{itemize}

\vspace{2mm}
\noindent \textbf{Remark 4.1.} From \eqref{equi}, it is evident that the pair $(\hat{\lambda},\hat{\omega})\in M\times N$ is uniquely determined by $\hat{p}$ due to \eqref{bcondition}, if $(\hat{p},\hat{\lambda},\hat{\omega})\in T$. Meanwhile, a map $\Upsilon:S\to  M\times N, \hat{p}\mapsto(\hat{\lambda}(\hat{p}),\hat{\omega}(\hat{p}))$ has been defined. If another set $T_0$ is defined to include the first components of all the elements in $T$, i.e.
\begin{equation}
\begin{split}
T_0= &~ \{q\in S:\exists(\eta,\xi)\in M\times N~\mathrm{s.t.}~[p,\mathcal{A}v]=
b_1(v,\eta)+b_2(v,\xi)+l(v),\forall v\in V\}\;\\
= & ~\{q\in S:(q,\eta,\xi)\in T\}\,,
\end{split}
\end{equation}
one can see that $((\hat{p},\hat{\lambda},\hat{\omega})\in T)$ is equivalent to that $\hat{p}\in T_0$, and $(\hat{\lambda},\hat{\omega})=(\hat{\lambda}(\hat{p}),\hat{\omega}(\hat{p}))$.

\subsection{Dual variational formulation}
\begin{theorem}
The dual variational problem of \eqref{func0} is expressed as
\begin{equation}\label{EVIdual}
(\bar{p},\bar{\lambda},\bar{\omega})=\arg\max_{(q,\eta,\xi)\in S\times M\times N}\left\{-\frac{1}{2}[q,q]+g_1(\eta)+g_2(\xi)-I_T(q,\eta,\xi)\right\}\,.
\end{equation}
\end{theorem}

\begin{proof}
From the Legendre-Fenchel transform
\begin{equation}
\frac{1}{2}a(v,v)=\max_{q\in S}~\left\{[q,\mathcal{A}v]-\frac{1}{2}[q,q]\right\}\,,\quad v\in V\,,
\end{equation}
the minimizing problem \eqref{func0} can be further expressed as
\begin{equation}
\begin{split}
\min_{v\in V}\max_{(q,\eta,\xi)\in S\times M\times N} \mathcal{L}_0(v,q,\eta,\xi)\,,\qquad\qquad\qquad\qquad\qquad&\\
\mathcal{L}_0(v,q,\eta,\xi)=\mathcal{L}(v,\eta,\xi)+[q,\mathcal{A}v]-\frac{1}{2}[q,q]-\frac{1}{2}a(v,v)\,,\quad
(v,q,\eta,\xi)\in V\times S\times M\times N\,.&
\end{split}
\end{equation}
Then one has
\begin{equation}
\min_{v\in V}~\mathcal{L}_0(v,q,\eta,\xi)=  -\frac{1}{2}[q,q]+g_1(\eta)+g_2(\xi)
+ \min_{v\in V}~\left\{[p,\mathcal{A}v]-l(v)-b_1(v,\eta)-b_2(v,\xi)\right\}\,.
\end{equation}
The definition of $T$ gives
\begin{equation}
\min_{v\in V}~\left\{[p,\mathcal{A}v]-l(v)-b_1(v,\eta)-b_2(v,\xi)\right\}=-I_T(q,\eta,\xi)\,.
\end{equation}
Then, the maximizing problem in \eqref{EVIdual} is obtained. The existence and uniqueness of maximizer $(\bar{p},\bar{\lambda},\bar{\omega})$ of the quadratic functional is also ensured.
\end{proof}

\begin{theorem}
The maximizer $(\bar{p},\bar{\lambda},\bar{\omega})$ for the dual maximizing variational problem \eqref{EVIdual} is the exact solution $({p},{\lambda},{\omega})$ of the EVI problem stated by \eqref{evic}, \eqref{equi} and \eqref{consti}.
\end{theorem}

\begin{proof}
Due to the existence and uniqueness of the minimizer of the dual maximizing problem \eqref{EVIdual}, all needed to prove is that $({p},{\lambda},{\omega})\in T$ is exactly the maximizer, which can be completed once the following inequality, an extreme condition
\begin{equation}
[q-p,\mathcal{A}u]-g_1(\eta-\lambda)-g_2(\xi-\omega)\geq 0\quad \forall(q,\eta,\xi)\in T
\end{equation}
is established.

Since $(q,\eta,\xi)\in T$, $(p,\lambda,\omega)\in T$ and $u\in V$,
\begin{equation*}
\begin{split}
[q-p,\mathcal{A}u] & =[b_1(u,\eta)+b_2(u,\xi)+l(v)]-[b_1(u,\lambda)+b_2(u,\omega)+l(v)]\\
& = b_1(u,\eta-\lambda)+b_2(u,\xi-\omega)\,,
\end{split}
\end{equation*}
one has
\begin{equation*}
[q-p,\mathcal{A}u]-g_1(\eta-\lambda)-g_2(\xi-\omega)
= [b_1(u,\eta-\lambda)-g_1(\eta-\lambda)]+[b_2(u,\xi-\omega)-g_2(\xi-\omega)]\geq 0\,,
\end{equation*}
incorporating \eqref{mix-2} and \eqref{mix-3}.
\end{proof}

With the introduction of a potential energy functional $\mathfrak{F}_p:V\to \mathbb{R}$ and a complementary energy functional $\mathfrak{F}_c:S\times L\times Q\to \mathbb{R}$,
\begin{equation}
\begin{split}
\mathfrak{F}_p(v)& =\frac{1}{2}a(v,v)-l(v)+j(v)\,,\quad v\in V\,,\\
\mathfrak{F}_c(q,\eta,\xi)& =\frac{1}{2}[q,q]-g_1(\eta)-g_2(\xi)\,,\quad (q,\eta,\xi)\in S\times L\times Q\,,
\end{split}
\end{equation}
the primal and dual variational formulations of the problem can be written as
\begin{equation}
\begin{split}
\mathfrak{F}_p(u)& =\min_{v\in K}~\mathfrak{F}_p(v)\,,\\
-\mathfrak{F}_c(p,\lambda,\omega)& =\max_{(q,\eta,\xi)\in T}~-\mathfrak{F}_c(q,\eta,\xi)\,,
\end{split}
\end{equation}
respectively.

\begin{theorem}
$(u,(p,\lambda,\omega))$ is a saddle point of ${\mathcal{L}}_0$ in $V\times (S\times M\times N)$, i.e.
\begin{equation}
\min_{v\in V}\max_{(q,\eta,\xi)\in S\times M\times N} \mathcal{L}_0(v,q,\eta,\xi)=\mathcal{L}_0(u,p,\lambda,\omega)=\max_{(q,\eta,\xi)\in S\times M\times N}\min_{v\in V} \mathcal{L}_0(v,q,\eta,\xi)\,.
\end{equation}
\end{theorem}

\begin{proof}
One has
\begin{equation*}
\begin{split}
& \min_{v\in V}\max_{(q,\eta,\xi)\in S\times M\times N} \mathcal{L}_0(v,q,\eta,\xi)=\min_{v\in K}~\mathfrak{F}_p(v)=\mathfrak{F}_p(u)\,,\\
& \max_{(q,\eta,\xi)\in S\times M\times N}\min_{v\in V} \mathcal{L}_0(v,q,\eta,\xi)=\max_{(q,\eta,\xi)\in T}~-\mathfrak{F}_c(q,\eta,\xi)=-\mathfrak{F}_c(p,\lambda,\omega)\,,
\end{split}
\end{equation*}
and it can be shown that
\begin{equation}\label{pippic2}
\mathfrak{F}_p(u)=\mathcal{L}_0(u,p,\lambda,\omega)=-\mathfrak{F}_c(p,\lambda,\omega)
\end{equation}
using the constitutive relation \eqref{consti} and the fact that $(u,(p,\lambda,\omega))\in V\times T$.
\end{proof}

\subsection{Introduction to a GCRE}
Considering the formulation of generalized constitutive relation error \eqref{gcre}, it can be obviously seen that $\mathcal{U}=K$ and $\mathcal{T}=T$ in this case, and
\begin{equation}
\begin{split}
& \Psi(\hat{u},(\hat{p},\hat{\lambda},\hat{\omega})) \\
= & \frac{1}{2}[\hat{p}-\mathcal{A}\hat{u},\hat{p}-\mathcal{A}\hat{u}]+b_1(\hat{u},\hat{\lambda})-g_1(\hat{\lambda})+b_2(\hat{u},\hat{\omega})-g_2(\hat{\omega})+j(\hat{u})\\
= & \left\{\frac{1}{2}a(\hat{u},\hat{u})+j(\hat{u})\right\}+\left\{\frac{1}{2}[\hat{p},\hat{p}]-g_1(\hat{\lambda})-g_2(\hat{\omega})\right\}\\
& -\left\{[\hat{p},\mathcal{A}\hat{u}]-b_1(\hat{u},\hat{\lambda})-b_2(\hat{u},\hat{\omega}) \right\}\geq 0\,,\quad (\hat{u},(\hat{p},\hat{\lambda},\hat{\omega}))\in K\times T \,.
\end{split}
\end{equation}
Then, one can express the generalized potential and complementary energy functionals, $\phi$ and $\phi^*$, and the bilinear form $\mathcal{B}$ as follows:
\begin{equation}\label{deffunc}
\begin{split}
\phi(v)& =\frac{1}{2}a({v},{v})+j({v})=\mathfrak{F}_p(v)+l(v)\,,\quad v\in V\,, \\
\phi^*(q,\eta,\xi)& =\frac{1}{2}[q,q]-g_1(\eta)-g_2(\xi)=\mathfrak{F}_c(q,\eta,\xi)\,,\quad (q,\eta,\xi)\in S\times L\times Q\,, \\
\mathcal{B}(v,(q,\eta,\xi))& =[q,\mathcal{A}v]-b_1(v,\eta)-b_2(v,\xi) \,,\quad
(v,(q,\eta,\xi))\in V\times(S\times L\times Q)\,.
\end{split}
\end{equation}

From the definitions of error functionals in \eqref{errdef}, it can be verified with $e=\hat{u}-u$ and $r=(\hat{p}-p,\hat{\lambda}-\lambda,\hat{\omega}-\omega)$ that
\begin{equation}\label{EVIf1}
\begin{split}
\bar{\phi}_{(u,(p,\lambda,\omega))}(\hat{u}-u)= & \left\{\frac{1}{2}a(\hat{u},\hat{u})+j(\hat{u})\right\}-\left\{\frac{1}{2}a({u},{u})+j({u})\right\} \\
& -\left\{[{p},\mathcal{A}\hat{u}-\mathcal{A}{u}]-b_1(\hat{u}-u,{\lambda})-b_2(\hat{u}-u,{\omega}) \right\}\\
= & ~\frac{1}{2}a(\hat{u}-u,\hat{u}-u)+\left\{b_1(\hat{u},\lambda)-g_1(\lambda)\right\}\\
& +\left\{b_2(\hat{u},\omega)-g_2(\omega)+j(\hat{u})\right\} \geq 0 \quad \forall\hat{u}\in K\,
\end{split}
\end{equation}
for $(\lambda,\omega)\in M\times N$, as well as
\begin{equation}\label{EVIf2}
\begin{split}
\bar{\phi}^*_{(u,(p,\lambda,\omega))}(\hat{p}-p,\hat{\lambda}-\lambda,\hat{\omega}-\omega)
= & ~\left\{\frac{1}{2}[\hat{p},\hat{p}]-g_1(\hat{\lambda})-g_2(\hat{\omega})\right\}-\left\{\frac{1}{2}[{p},{p}]-g_1({\lambda})-g_2({\omega})\right\} \\
& -\left\{[\hat{p}-p,\mathcal{A}{u}]-b_1(u,\hat{\lambda}-\lambda)-b_2(u,\hat{\omega}-\omega) \right\}\\
= & ~\frac{1}{2}[\hat{p}-p,\hat{p}-p]+\left\{b_1({u},\hat{\lambda})-g_1(\hat{\lambda})\right\}\\
& +\left\{b_2({u},\hat{\omega})-g_2(\hat{\omega})+j({u})\right\} \geq 0 \quad \forall (\hat{p},\hat{\lambda},\hat{\omega})\in T\,,
\end{split}
\end{equation}
for $u \in K$. Note that the constitutive relations \eqref{consti} are used repeatedly.  

\subsection{Some properties of the GCRE}

\begin{theorem}
The GCRE can be represented as the sum of two error functionals of the admissible solutions $(\hat{u},(\hat{p},\hat{\lambda},\hat{\omega}))$, i.e.
\begin{equation}\label{EVIplit}
\Psi(\hat{u},(\hat{p},\hat{\lambda},\hat{\omega}))=  \bar{\phi}_{(u,(p,\lambda,\omega))}(\hat{u}-u)+\bar{\phi}^*_{(u,(p,\lambda,\omega))}(\hat{p}-p,\hat{\lambda}-\lambda,\hat{\omega}-\omega)\,
\end{equation}
which provides a strict upper bound for the global energy error of either of the admissible solutions $(\hat{u},\hat{p})$:
\begin{equation}\label{uppbs}
\Psi(\hat{u},(\hat{p},\hat{\lambda},\hat{\omega}))\geq
\begin{cases}
\bar{\phi}_{(u,(p,\lambda,\omega))}(\hat{u}-u) & \geq \frac{1}{2}a(u-\hat{u},u-\hat{u})\,,\\
\bar{\phi}^*_{(u,(p,\lambda,\omega))}(\hat{p}-p,\hat{\lambda}-\lambda,\hat{\omega}-\omega)
& \geq \frac{1}{2}[p-\hat{p},p-\hat{p}]\,,
\end{cases}
\end{equation}
\end{theorem}

\begin{proof}
Eq.\ \eqref{relation} gives the following identity for this EVI problem:
\begin{equation}
\begin{split}
\Psi(\hat{u},(\hat{p},\hat{\lambda},\hat{\omega}))= & \bar{\phi}_{(u,(p,\lambda,\omega))}(\hat{u}-u)+\bar{\phi}^*_{(u,(p,\lambda,\omega))}(\hat{p}-p,\hat{\lambda}-\lambda,\hat{\omega}-\omega)\\
& -\mathcal{B}(\hat{u}-u,(\hat{p}-p,\hat{\lambda}-\lambda,\hat{\omega}-\omega)),
\end{split}
\end{equation}
relating the errors of admissible solutions to the GCRE.

Since $\hat{u}-u\in V$, $(\hat{\lambda},\lambda)=(\hat{\lambda}(\hat{p}),\hat{\lambda}({p}))$ and
$(\hat{\omega},\omega)=(\hat{\omega}(\hat{p}),\hat{\omega}({p}))$ from Remark 4.1, one has
\begin{equation*}
\begin{split}
& \mathcal{B}(\hat{u}-u,(\hat{p},\hat{\lambda},\hat{\omega}))=[ \hat{p},\mathcal{A}(\hat{u}-u) ]-b_1(\hat{u}-u,\hat{\lambda}(\hat{p}))-b_2(\hat{u}-u,\hat{\omega}(\hat{p}))=l(\hat{u}-u)\,,\\
& \mathcal{B}(\hat{u}-u,(p,\lambda,\omega))=[ {p},\mathcal{A}(\hat{u}-u) ]-b_1(\hat{u}-u,\hat{\lambda}({p}))-b_2(\hat{u}-u,\hat{\omega}({p}))=l(\hat{u}-u)\,,
\end{split}
\end{equation*}
due to \eqref{equi} and then
\begin{equation}
\mathcal{B}(\hat{u}-u,(\hat{p}-p,\hat{\lambda}-\lambda,\hat{\omega}-\omega))=0\,.
\end{equation}
Thus \eqref{EVIplit} is obtained.

From \eqref{EVIf1} and \eqref{EVIf2}, it is obvious that
\begin{equation*}
\bar{\phi}_{(u,(p,\lambda,\omega))}(\hat{u}-u) \geq \frac{1}{2}a(u-\hat{u},u-\hat{u})\geq 0
\end{equation*}
and
\begin{equation*}
\bar{\phi}^*_{(u,(p,\lambda,\omega))}(\hat{p}-p,\hat{\lambda}-\lambda,\hat{\omega}-\omega)
\geq \frac{1}{2}[p-\hat{p},p-\hat{p}]\geq 0\,,
\end{equation*}
which give \eqref{uppbs}
\end{proof}

\begin{theorem}
The GCRE for the present EVI problem can be represented as the sum of the potential energy subject to an admissible field $\hat{u}\in K$ and the complementary energy subject to an admissible field $(\hat{p},\hat{\lambda},\hat{\omega})\in T$, i.e.
\begin{equation}\label{pippic3}
\Psi(\hat{u},(\hat{p},\hat{\lambda},\hat{\omega}))=\mathfrak{F}_p(\hat{u})+\mathfrak{F}_c(\hat{p},\hat{\lambda},\hat{\omega})\quad (\hat{u},(\hat{p},\hat{\lambda},\hat{\omega}))\in K\times T\,.
\end{equation}
\end{theorem}

\begin{proof}
The definition of $T$ gives that for $(\hat{p},\hat{\lambda},\hat{\omega})\in T$,
\begin{equation*}
\mathcal{B}(v,(\hat{p},\hat{\lambda},\hat{\omega}))=l(v)\quad \forall v\in V\,.
\end{equation*}
Then \eqref{pippic3} comes out naturally from the definitions of $\Psi$, $\mathfrak{F}_p$ and $\mathfrak{F}_c$.
\end{proof}

\subsection{A frictional contact model problem}

\begin{figure}
  \centering
  \includegraphics[scale=0.6]{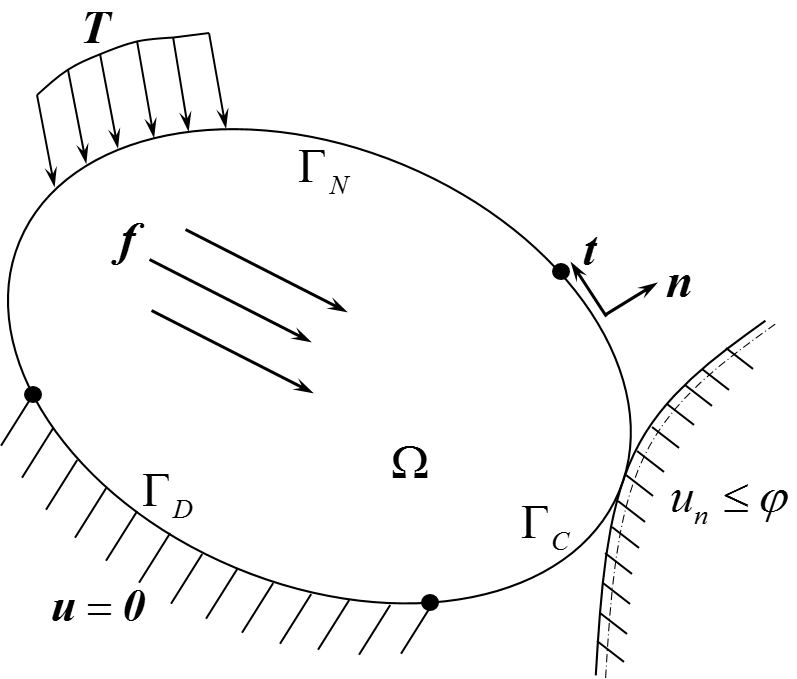}
  \caption{A frictional contact problem of a linear elastic body}\label{contact}
\end{figure}

Let us consider a Tresca-type frictional contact problem of a linear elastic body defined in $\Omega\subset \mathbb{R}^3$ with small deformation, as shown in Figure \ref{contact}, which is to find $(\boldsymbol{u},\boldsymbol{\sigma})$ such that
\begin{equation}
\begin{split}
\mathrm{div}\boldsymbol{\boldsymbol{\sigma}}+\boldsymbol{f}=\boldsymbol{0}\,,\quad \boldsymbol{\sigma}=\boldsymbol{K}:\nabla_S \boldsymbol{u}\quad \mathrm{in}\;\Omega\,,&\\
\boldsymbol{\sigma}\boldsymbol{n}=\boldsymbol{T}\quad \mathrm{on}\;\Gamma_N\,,\quad \boldsymbol{u}=\boldsymbol{0}\quad \mathrm{on}\;\Gamma_D\,,&\\
u_n\leq \varphi\,,\quad \sigma_n\leq 0\,,\quad
\sigma_n(u_n-\varphi)=0 \quad \mathrm{on}\;\Gamma_C\,,&\\
|\boldsymbol{\sigma}_{t}|\leq s\;\mathrm{with}\;
\begin{cases}
|\boldsymbol{\sigma}_{t}|< s\,\Rightarrow\,\boldsymbol{u}_{t}=\boldsymbol{0}\\
|\boldsymbol{\sigma}_{t}|= s\,\Rightarrow\,\exists\zeta\in\mathbb{R}^+:\boldsymbol{u}_{t}=-\zeta\boldsymbol{\sigma}_{t}
\end{cases}
\quad \mathrm{on}\;\Gamma_C\,,&
\end{split}
\end{equation}
where $\boldsymbol{K}$ is Hooke's stiffness tensor for isotropic linear elastic materials, and
\begin{equation}\label{ntau}
\begin{split}
& u_n=\boldsymbol{n}\cdot\boldsymbol{u}\,,\quad \boldsymbol{u}_{t}=\boldsymbol{u}-u_n \boldsymbol{n}\,,\\
& \sigma_n=(\boldsymbol{\sigma}\boldsymbol{n})\cdot\boldsymbol{n}\,,\quad \boldsymbol{\sigma}_{t}=\boldsymbol{\sigma}\boldsymbol{n}-\sigma_n \boldsymbol{n}\,.
\end{split}
\end{equation}

Alternatively, this problem can be expressed in the form of \eqref{func0} or \eqref{func} by taking
\begin{equation}
\begin{split}
V=\{\boldsymbol{v}\in [H^1(\Omega)]^3:\boldsymbol{v}|_{\Gamma_D}=\boldsymbol{0}\}\,,\quad S=\{\boldsymbol{\tau}\in [H(\mathrm{div},\Omega)]^{3}:\boldsymbol{\tau}^T=\boldsymbol{\tau}\}\,,&\\
L=L^2(\Gamma_C)\,,\quad M=\{\eta\in L:\eta\leq 0\;\mathrm{a.e.~on}\;\Gamma_C\}\,,&\\
Q=\{\boldsymbol{\xi}\in[L^2(\Gamma_C)]^2:\boldsymbol{\xi}\perp \boldsymbol{n}\}\,,\quad N=\{\boldsymbol{\xi}\in Q:|\boldsymbol{\xi}|\leq s\;\mathrm{a.e.~on}\;\Gamma_C\}\,,&\\
[\boldsymbol{\sigma},\boldsymbol{\tau}]=\langle\boldsymbol{\sigma},\boldsymbol{K}^{-1}:\boldsymbol{\tau}\rangle\,,\quad (\boldsymbol{\sigma},\boldsymbol{\tau})\in S\times S\,,&\\
a(\boldsymbol{u},\boldsymbol{v})=\langle \boldsymbol{K}:\nabla_S \boldsymbol{u},\nabla_S \boldsymbol{v}\rangle\,,\quad (\boldsymbol{u},\boldsymbol{v})\in V\times V\,,&\\
l(\boldsymbol{v})=\int_\Omega\boldsymbol{f}\cdot\boldsymbol{v}+\int_{\Gamma_N}\boldsymbol{T}\cdot\boldsymbol{v}\,,\quad j(\boldsymbol{v})=\int_{\Gamma_C}s|\boldsymbol{v}_{t}|\,,\quad \boldsymbol{v}\in V\,,&\\
b_1(\boldsymbol{v},\eta)=\int_{\Gamma_C}v_n \eta\,,\quad g_1(\eta)=\int_{\Gamma_C}\varphi\eta\,,\quad (\boldsymbol{v},\eta)\in V\times L\,, &\\
b_2(\boldsymbol{v},\boldsymbol{\xi})=\int_{\Gamma_C}\boldsymbol{v}_{t}\cdot\boldsymbol{\xi}\,,\quad g_2(\boldsymbol{\xi})=0\,,\quad (\boldsymbol{v},\boldsymbol{\xi})\in V\times Q\,, &
\end{split}
\end{equation}
where the definitions of $v_n$ and $\boldsymbol{v}_{t}$ are similar to those of $u_n$ and $\boldsymbol{u}_{t}$ in \eqref{ntau}, respectively.

The admissible solutions are given as follows:
\begin{equation}
\begin{split}
\hat{\boldsymbol{u}}\in K & =\left\{\boldsymbol{v}\in V:\int_{\Gamma_C}(v_n-\varphi)\eta\geq 0\,,\;\forall \eta \in M\right\}\,,\\
(\hat{\boldsymbol{\sigma}},\hat{\sigma}_n,\hat{\boldsymbol{\sigma}}_{t})\in T & =
\left\{(\boldsymbol{\tau},\eta,\boldsymbol{\xi})\in S\times M\times N:\langle\boldsymbol{\tau},\nabla_S\boldsymbol{v}\rangle=
\int_{\Gamma_C}\left(v_n\eta+\boldsymbol{v}_{t}\cdot\boldsymbol{\xi}\right)+l(\boldsymbol{v})\,,\;\forall \boldsymbol{v}\in V\right\},
\end{split}
\end{equation}
and the definition of $T$ is alternatively written in the following form using Green's theorem:
\begin{equation}\label{defT}
\begin{split}
T=\{(\boldsymbol{\tau},\eta,\boldsymbol{\xi})\in S\times M\times N:\mathrm{div}\boldsymbol{\tau}+\boldsymbol{f}=0\;\mathrm{a.e.~in}\;\Omega\,,\;&\\
\boldsymbol{\tau}\boldsymbol{n}=\boldsymbol{T}\;\mathrm{a.e.~on}\;\Gamma_N\,,\;\left(\tau_n,\,\boldsymbol{\tau}_t\right)=(\eta,\boldsymbol{\xi})\;\mathrm{a.e.~on}\;\Gamma_C\}\,.
\end{split}
\end{equation}

In this case, the generalized potential and complementary energy functionals are defined as:
\begin{equation}
\begin{split}
\phi(\boldsymbol{v}) & =\frac{1}{2}\langle \boldsymbol{K}:\nabla_S \boldsymbol{v},\nabla_S \boldsymbol{v}\rangle+\int_{\Gamma_C}s|\boldsymbol{v}_{t}|\,,\quad \boldsymbol{v}\in V\,,\\
\phi^*(\boldsymbol{\tau},\eta,\boldsymbol{\xi}) & =\frac{1}{2}\langle\boldsymbol{\tau},\boldsymbol{K}^{-1}:\boldsymbol{\tau}\rangle-\int_{\Gamma}\varphi\eta\,,\quad (\boldsymbol{\tau},\eta,\boldsymbol{\xi})\in S\times L\times Q\,,\\
\mathcal{B}(\boldsymbol{v},(\boldsymbol{\tau},\eta,\boldsymbol{\xi})) & =\langle\boldsymbol{\tau},\nabla_S \boldsymbol{v}\rangle-\int_{\Gamma_C}\left(v_n\eta+\boldsymbol{v}_{t}\cdot\boldsymbol{\xi}\right)\,,\quad (\boldsymbol{v},(\boldsymbol{\tau},\eta,\boldsymbol{\xi}))\in V\times(S\times L\times Q)\,.
\end{split}
\end{equation}

Furthermore, the GCRE provides an upper bound of global energy error of admissible displacement field $\hat{\boldsymbol{u}}$ or admissible stress field $\hat{\boldsymbol{\sigma}}$, i.e
\begin{equation}
\begin{split}
\Psi(\hat{\boldsymbol{u}},(\hat{\boldsymbol{\sigma}},\hat{\sigma}_n,\hat{\boldsymbol{\sigma}}_{t}))=&~\mathfrak{F}_p(\hat{\boldsymbol{u}})+\mathfrak{F}_c(\hat{\boldsymbol{\sigma}},\hat{\sigma}_n,\hat{\boldsymbol{\sigma}}_{t})\\
= &~ \bar{\phi}_{(\boldsymbol{u},(\boldsymbol{\sigma},\sigma_n,\boldsymbol{\sigma}_t))}(\hat{\boldsymbol{u}}-\boldsymbol{u})+\bar{\phi}^*_{(\boldsymbol{u},(\boldsymbol{\sigma},\sigma_n,\boldsymbol{\sigma}_t))}(\hat{\boldsymbol{\sigma}}-\boldsymbol{\sigma},\hat{\sigma}_n-\sigma_n,\hat{\boldsymbol{\sigma}}_t-\boldsymbol{\sigma}_t)\\
\geq &~
\begin{cases}
\frac{1}{2}a(\boldsymbol{u}-\hat{\boldsymbol{u}},\boldsymbol{u}-\hat{\boldsymbol{u}})\,,\\
\frac{1}{2}[\boldsymbol{\sigma}-\hat{\boldsymbol{\sigma}},\boldsymbol{\sigma}-\hat{\boldsymbol{\sigma}}]\,.
\end{cases}
\end{split}
\end{equation}

\section{Application to \emph{a posteriori} error estimation}

The key idea about the application of GCRE to \emph{a posteriori} error estimation is to construct admissible solutions based on the numerical solution to the problem in question. For the convenience of expression, the frictional contact problem in Subsection 4.5 is taken as an example herein.

When a displacement field solution $\boldsymbol{u}_h$ is obtained using a specific numerical method, such as the conforming or non-conforming finite element method, the discontinuous Galerkin method, and the mixed finite element method with the displacement field as unknown, an admissible displacement field $\hat{\boldsymbol{u}}_h \in K$ can be constructed via post-processing, even if $\boldsymbol{u}_h \not\in K$. In this process, $\hat{\boldsymbol{u}}_h$ is usually taken close enough to $\boldsymbol{u}_h$, and can also act as a solution to the problem under the mesh characterized by size $h$. Then one has
\begin{equation}
a(\boldsymbol{u}-\boldsymbol{u}_h,\boldsymbol{u}-\boldsymbol{u}_h)\approx a(\boldsymbol{u}-\hat{\boldsymbol{u}}_h,\boldsymbol{u}-\hat{\boldsymbol{u}}_h)\,.
\end{equation}

The admissible stress-traction solution $(\hat{\boldsymbol{\sigma}}_h,\hat{\sigma}_{nh},\hat{\boldsymbol{\sigma}}_{th})$ can be constructed using some available recovering techniques to obtain equilibrated stress fields from the displacement solution $\boldsymbol{u}_h$. These recovering techniques of equilibrated stress fields includes those based on element-wise Neumann problems (such as the element equilibrium technique (EET) \cite{pled2011techniques}) and those based on subdomain-wise problems (such as the traction-free method based on the partition of unity \cite{gallimard2009constitutive}), see \cite{ladeveze2005mastering,ainsworth2011posteriori} for more references.

Then, one has the GCRE as a strictly upper-bounding estimator:
\begin{equation}
a(\boldsymbol{u}-\hat{\boldsymbol{u}}_h,\boldsymbol{u}-\hat{\boldsymbol{u}}_h)\leq 2~\Psi(\hat{\boldsymbol{u}}_h,(\hat{\boldsymbol{\sigma}}_h,\hat{\sigma}_{nh},\hat{\boldsymbol{\sigma}}_{th}))\,.
\end{equation}
Usually, the recovering techniques ensure that the asymptotic behavior of CRE is verified \cite{ladeveze2005mastering}, i.e. there exists a constant $c >0$ independent of $h$ such that
\begin{equation}
\Psi(\hat{\boldsymbol{u}}_h,(\hat{\boldsymbol{\sigma}}_h,\hat{\sigma}_{nh},\hat{\boldsymbol{\sigma}}_{th})) \leq c ~ a(\boldsymbol{u}-\boldsymbol{u}_h,\boldsymbol{u}-\boldsymbol{u}_h)\,.
\end{equation}

It is worth figuring out that the CRE for visco-plasticity problems in \cite{ladeveze2016constitutive}, which acts as \emph{a posteriori} global error estimator, could be considered as a specific case of the proposed GCRE. The CREs for Coulomb-type frictional contact problems, elastic body to elastic body in \cite{becheur2008posteriori,louf2003constitutive} and elastic body to rigid foundation in \cite{coorevits2001posteriori}, could also be classified into the GCRE, coinciding with the abstract EVI form in Section 4.

\section{Conclusions}

A generalized constitutive relation error is proposed in an analogous but generalized formulation of Fenchel-Young inequality, a natural result of Legendre-Fenchel duality. It has been shown that the GCRE provides strict upper bounds of the global errors of admissible solutions to a class of elliptic variational inequalities. The GCRE is believed to be applicable to a wide range of convex variational problems and the corresponding \emph{a posteriori} error estimation of the numerical solutions.

\section*{Acknowledgement}

Some of the work was carried out during the first author's visit to Department of Mathematics at the University of Iowa under the financial support of Tsinghua Scholarship for Overseas Graduate Studies under Grant No. 2015011. This work was also supported by the National Natural Science Foundation of China under Grant No. 51378294.


\bibliographystyle{abbrv}
\bibliography{refs}

\begin{thebibliography}{10}

\bibitem{ainsworth2011posteriori}
M.~Ainsworth and J.~T. Oden.
\newblock {\em A posteriori error estimation in finite element analysis},
  volume~37.
\newblock John Wiley \& Sons, 2011.

\bibitem{atkinson2005theoretical}
K.~Atkinson and W.~Han.
\newblock {\em Theoretical numerical analysis: a functional analysis framework,
  3rd edition}.
\newblock Springer, 2009.

\bibitem{becheur2008posteriori}
A.~Becheur, A.~Tahakourt, and P.~Coorevits.
\newblock An a posteriori error indicator for $\mathrm{C}$oulomb's frictional
  contact.
\newblock {\em Mechanics Research Communications}, 35(8):562--575, 2008.

\bibitem{borwein2010convex}
J.~M. Borwein and A.~S. Lewis.
\newblock {\em Convex analysis and nonlinear optimization: theory and
  examples}.
\newblock Springer Science \& Business Media, 2010.

\bibitem{chamoin2009strict}
L.~Chamoin and P.~Ladev{\`e}ze.
\newblock Strict and practical bounds through a non-intrusive and goal-oriented
  error estimation method for linear viscoelasticity problems.
\newblock {\em Finite Elements in Analysis and Design}, 45(4):251--262, 2009.

\bibitem{ciarlet1988mathematical}
P.~G. Ciarlet.
\newblock {\em Mathematical Elasticity, volume I: Three-dimensional
  Elasticity}.
\newblock North-Holland, Amsterdam, 1988.

\bibitem{ciarlet2012new}
P.~G. Ciarlet, G.~Geymonat, and F.~Krasucki.
\newblock A new duality approach to elasticity.
\newblock {\em Mathematical Models and Methods in Applied Sciences},
  22(01):1150003, 2012.

\bibitem{coorevits2001posteriori}
P.~Coorevits, P.~Hild, and M.~Hjiaj.
\newblock A posteriori error control of finite element approximations for
  $\mathrm{C}$oulomb's frictional contact.
\newblock {\em SIAM Journal on Scientific Computing}, 23(3):976--999, 2001.

\bibitem{ekeland1976convex}
I.~Ekeland and R.~Temam.
\newblock {\em Convex analysis and variational problems}.
\newblock SIAM, 1976.

\bibitem{gallimard2009constitutive}
L.~Gallimard.
\newblock A constitutive relation error estimator based on traction-free
  recovery of the equilibrated stress.
\newblock {\em International Journal for Numerical Methods in Engineering},
  78(4):460--482, 2009.

\bibitem{guo2015goal}
M.~Guo and H.~Zhong.
\newblock Goal-oriented error estimation for beams on elastic foundation with
  double shear effect.
\newblock {\em Applied Mathematical Modelling}, 39(16):4699--4714, 2015.

\bibitem{ladeveze2008strict}
P.~Ladev{\`e}ze.
\newblock Strict upper error bounds on computed outputs of interest in
  computational structural mechanics.
\newblock {\em Computational Mechanics}, 42(2):271--286, 2008.

\bibitem{ladeveze2010calculation}
P.~Ladev\`{e}ze and L.~Chamoin.
\newblock Calculation of strict error bounds for finite element approximations
  of non-linear pointwise quantities of interest.
\newblock {\em International Journal for Numerical Methods in Engineering},
  84(13):1638--1664, 2010.

\bibitem{ladeveze2016constitutive}
P.~Ladev{\`e}ze and L.~Chamoin.
\newblock The constitutive relation error method: A general verification tool.
\newblock In {\em Verifying Calculations-Forty Years On}, pages 59--94.
  Springer, 2016.

\bibitem{ladeveze1983error}
P.~Ladev\`{e}ze and D.~Leguillon.
\newblock Error estimate procedure in the finite element method and
  applications.
\newblock {\em SIAM Journal on Numerical Analysis}, 20(3):485--509, 1983.

\bibitem{ladeveze2005mastering}
P.~Ladev{\`e}ze, J.~P. Pelle, F.~F. Ling, E.~F. Gloyna, and W.~H. Hart.
\newblock {\em Mastering calculations in linear and nonlinear mechanics}.
\newblock Springer, 2005.

\bibitem{ladeveze2013new}
P.~Ladev\`{e}ze, F.~Pled, and L.~Chamoin.
\newblock New bounding techniques for goal-oriented error estimation applied to
  linear problems.
\newblock {\em International Journal for Numerical Methods in Engineering},
  93(13):1345--1380, 2013.

\bibitem{ladeveze1999local}
P.~Ladev\`{e}ze, P.~Rougeot, P.~Blanchard, and J.~Moreau.
\newblock Local error estimators for finite element linear analysis.
\newblock {\em Computer Methods in Applied Mechanics and Engineering},
  176(1):231--246, 1999.

\bibitem{louf2003constitutive}
F.~Louf, J.-P. Combe, and J.-P. Pelle.
\newblock Constitutive error estimator for the control of contact problems
  involving friction.
\newblock {\em Computers \& structures}, 81(18):1759--1772, 2003.

\bibitem{panetier2009strict}
J.~Panetier, P.~Ladev{\`e}ze, and F.~Louf.
\newblock Strict bounds for computed stress intensity factors.
\newblock {\em Computers \& Structures}, 87(15):1015--1021, 2009.

\bibitem{pled2011techniques}
F.~Pled, L.~Chamoin, and P.~Ladev\`{e}ze.
\newblock On the techniques for constructing admissible stress fields in model
  verification: Performances on engineering examples.
\newblock {\em International Journal for Numerical Methods in Engineering},
  88(5):409--441, 2011.

\bibitem{wang2015unified}
L.~Wang and H.~Zhong.
\newblock A unified approach to strict upper and lower bounds of quantities in
  linear elasticity based on constitutive relation error estimation.
\newblock {\em Computer Methods in Applied Mechanics and Engineering},
  286:332--353, 2015.

\end{thebibliography}

\end{document}